\documentclass[11pt]{scrartcl}

\newif\iffinal
\finalfalse
\finaltrue

\usepackage[utf8]{inputenc}

\usepackage{tikz}
\tikzset{>=stealth}

\usepackage[l2tabu, orthodox]{nag}

\usepackage[
	pdfusetitle,
	pagebackref,
	colorlinks=true,
	pdfpagemode=UseNone,
	urlcolor=blue,
	linkcolor=blue,
	citecolor=blue,
	pdfstartview=FitH
]{hyperref}
\hypersetup{pdfauthor={Klas Modin, Olivier Verdier}}

\usepackage{graphicx}

\usepackage[numbers]{natbib} 

\iffinal%
\else
\usepackage[notref,notcite%
]{showkeys}
\usepackage{rotating}

\fi%

\definecolor{weakorange}{rgb}{1,.9,.2}
\usepackage[textsize=footnotesize,backgroundcolor=weakorange,
\iffinal
disable,
\fi
]{todonotes}

\usepackage{amsmath,amsthm,amssymb}

\usepackage{eucal}

\usepackage{mathtools}
\mathtoolsset{centercolon,showonlyrefs=false}

\usepackage{aliascnt}
\usepackage{amsthm}
\newcommand{\newtheoremalias}[2]{%
\newaliascnt{#1}{theorem}%
\newtheorem{#1}[#1]{#2}%
\aliascntresetthe{#1}%
}

\theoremstyle{plain}
\newtheorem{theorem}{Theorem}[section]

\newtheoremalias{proposition}{Proposition}

\newtheoremalias{corollary}{Corollary}

\newtheoremalias{lemma}{Lemma}

\theoremstyle{definition}
\newtheoremalias{definition}{Definition}

\theoremstyle{definition}
\newtheoremalias{example}{Example}

\newtheoremalias{remark}{Remark}

\usepackage[affil-sl]{authblk} 

\title{Integrability of Nonholonomically Coupled Oscillators}
\author[1,2]{Klas Modin\thanks{klas.modin@chalmers.se}}
\author[3]{Olivier Verdier\thanks{olivier.verdier@math.uib.no}}

\affil[1]{
	Department of Mathematics, University of Toronto, Canada
}
\affil[2]{
	Department of Mathematical Sciences, Chalmers University of Technology, Sweden
}
\affil[3]{
	Department of Mathematics,
	University of Bergen,
	Norway
}

\newcommand{\ii}{\mathrm{i}}
\newcommand{\pair}[1]{\left\langle #1 \right\rangle}

\providecommand{\vect}[1]{\boldsymbol{#1}}

\newcommand{\inv}{^{-1}}

\newcommand{\ud}{\mathrm{d}}

\newcommand{\pd}{\partial}

\newcommand{\R}{{\mathbb R}}
\newcommand{\Z}{{\mathbb Z}}

\providecommand{\SO}{\mathrm{SO}}
\providecommand{\so}{\mathfrak{so}}

\newcommand*\px{p_1}
\newcommand*\py{p_2}

\newcommand*\matA{\mathsf{A}}
\newcommand*\matB{\mathsf{B}}
\newcommand*\matAbar{\overline{\mathsf{A}}}
\newcommand*\Torus{\R/\Z}
\newcommand*\perPhi{\flow(1)}

\newcommand*\Man{\mathcal{M}}

\newcommand*\thera{\tau}

\newcommand*\flow{\Phi}

\begin{document}

\maketitle

\begin{abstract}
	We study a family of nonholonomic mechanical systems.
	These systems consist of harmonic oscillators coupled through nonholonomic constraints.
	In particular, the family includes the so called \emph{contact oscillator}, which has been used as a test problem for numerical methods for nonholonomic mechanics.
	Furthermore, the systems under study constitute simple models for continuously variable transmission gearboxes.

	The main result is that each system in the family is integrable reversible with respect to the canonical reversibility map on the cotangent bundle.
	By using reversible Kolmogorov--Arnold--Moser theory, we then establish preservation of invariant tori for reversible perturbations.
	This result explains previous numerical observations, that some discretisations of the contact oscillator have favourable structure preserving properties.
\end{abstract}

\section{Introduction} 
\label{sec:introduction}

\noindent
In this paper we study a family of continuous dynamical systems consisting of nonholonomically coupled oscillators.
The main result is that the family is integrable and that integrability is stable under reversible perturbations.
Although numerical discretisations are not discussed in the paper, our main motivation is to obtain a rigorous explanation of the numerical observations by \citet{McPe2006}, that some discretisations for nonholonomic problems have structure preserving properties.

We continue the introduction with some background in Hamiltonian, nonholonomic, and reversible dynamics.
In \autoref{sec:main_results} we give a presentation of the family of problems under study, and we state the main results.
These results are then proved in \autoref{sec:proofs}.

The concept of Liouville--Arnold integrability is fundamental in the study of continuous Hamiltonian dynamical systems.
Through Kolmogorov--Arnold--Moser (KAM) theory, Liouville--Arnold integrability implies a foliation of the phase space in invariant tori, not only for the system itself, but also for ``nearby'' Hamiltonian systems (see \citet[Ch.~\!10]{Ar1989} and references therein).

By backward error analysis a numerical discretisation method for a continuous dynamical system can, at least formally, be interpreted as the exact solution of a nearby system~\cite[Ch.~\!IX]{HaLuWa2006}.
Roughly speaking, the KAM result then states that symplectic methods applied to Liouville--Arnold integrable systems preserve perturbed invariant tori~\cite{Sh1999b,Sh2000}.
This explains the excellent performance of symplectic methods for Hamiltonian systems.

In nonholonomic mechanics, there is no corresponding KAM theory, so the foliation in invariant tori are in general not stable with respect to perturbations (see the discussion in the end of \cite[\S~\!5.4.2]{ArKoNeKh2010}).
Nevertheless, integrability of nonholonomic systems is a well established research topic. 
In the current literature there are three main techniques to obtain complete integrability of nonholonomic systems.
One technique uses reduction by symmetry, and examples can be found in \cite{BlMaZe2009}.
Next, the special case of systems of Chaplygin type is treated using a nonholonomic version of the Hamilton--Jacobi theory in \cite{IgDaLeMa2008,LeMaMa2010,OhFeBlZe2011}.
Finally, one can appeal to the existence of an invariant measure, which is the technique used in \cite{Ko2002} and \cite[\S~\!5.4]{ArKoNeKh2010}.
All these techniques depend on some special property of the nonholonomic system at hand, and our case is no exception.
Indeed, the essential property we use is that the underlying differential equation is a non-autonomous linear system whose fundamental matrix belongs to~$\SO(n)$.

For discussions on numerical discretisations of nonholonomic problems we refer to~\cite{McPe2006,FeIgMa2008,FeIgMa2009,KoMaSu2010,KoMaDiFe2010}, to the monograph by \citet[\S~\!7]{Co2002} and to the references therein.
Due to the lack of KAM theory for nonholonomic mechanics, it is unclear which structure to preserve when discretising nonholonomic systems in order for the modified differential equation to remain integrable.

Let $Q$ be a smooth configuration manifold of dimension~$n$.
The corresponding phase space is given by the cotangent bundle $T^{*}Q$.
Recall that $T^{*}Q$ comes with a canonical symplectic form, which, in local canonical coordinates $(q_{i},p_{i})$, is given by $\omega = \sum_i \ud q_{i}\wedge\ud p_{i}$.
In addition, $T^{*}Q$ comes with a canonical reversibility map (as does any vector bundle), namely the bundle map $\rho\colon T^{*}Q\to T^{*}Q$ given in local coordinates by $(q_{i},p_{i}) \mapsto (q_{i},-p_{i})$.
A diffeomorphism $\psi$ of $T^*Q$ is said to be \emph{reversible} if $\rho\circ\psi = \psi^{-1}\circ\rho$.
A continuous dynamical system on $T^{*}Q$ with flow map $\varphi^{t}$ is called reversible if the flow map is reversible.
If $X$ is the vector field that generates $\varphi^{t}$, then the system is reversible if and only if $T\rho\circ X = -X\circ\rho$, or equivalently, $\rho_{*}X = -X$, where $\rho_*X$ is the push-forward of the vector field (see, e.g.,~\cite[\S\!~4.3]{MaRa1999}).
A reversible system is called \emph{integrable reversible} if $T^{*}Q$ can be foliated into invariant tori, which are also invariant with respect to the reversibility map~$\rho$.
That is, if there exists a reversible change of coordinates into action--angle variables~\cite[Def.\!~XI.1.1]{HaLuWa2006}.

There are various ways to define nonholonomic systems.
For an overview, see the monograph by \citet[Ch.\!~5]{Bl2003}.
In our setting, a nonholonomic system is defined by a triple $(Q,\mathcal{D},H)$, where~$Q$ is a configuration manifold, $\mathcal{D}$ is a regular distribution on $Q$, i.e., a vector subbundle of $TQ$, and $H$ is a Hamiltonian, i.e., a smooth function on $T^{*}Q$.
The constraints are given by $\dot{\vect{q}}\in \mathcal{D}$.
Equivalently, if $\mathcal{D}^{*}\subset T^{*}Q$ is the annihilator of $\mathcal{D}$, then the constraints are $\pair{\mathcal{D}^{*},\dot{\vect{q}}} = 0$.
If the one-forms $\tau_{1},\ldots,\tau_{k}$ span the codistribution $\mathcal{D}^{*}$, then the Hamiltonian version of the Lagrange--d'Alembert principle yields the governing equations (see \cite[\S\!~5.2]{Bl2003})
\begin{equation}\label{eq:nonhol_system_general}
	\begin{split}
		\dot q_{i} &= \frac{\pd H}{\pd p_{i}}, \\
		\dot p_{i} &= -\frac{\pd H}{\pd q_{i}} + \sum_{j=1}^{k} \lambda_j \tau_{ij}, \\
		0 &= \sum_{i}\tau_{ij} \frac{\pd H}{\pd p_i}, \quad j=1,\ldots,k,
	\end{split}
\end{equation}
where the quantities $\tau_{ij}$ are defined by $\tau_{j} = \sum_i \tau_{ij}\ud q_{i}$, and $\lambda_{j}$ are Lagrange multipliers.

We assume that the differential algebraic equation~\eqref{eq:nonhol_system_general} defines an underlying ordinary differential equation (ODE) on the manifold
\begin{equation}\label{eq:cobundle_distribution}
	\mathcal{M} = \Big\{\, (\vect{q},\vect{p})\in T^{*}Q; \Big\langle{\mathcal{D}^{*},\frac{\pd H}{\pd \vect{p}}}\Big\rangle=0 \,\Big\}.
\end{equation}

Notice that if $H$ is regular, which we assume, then $\mathcal{M}$ is a fibre bundle over $Q$.
$\mathcal{M}$ is a vector subbundle of $T^{*}Q$ if and only if $H$ is quadratic in~$\vect{p}$ (which is typically the case).
It is closed under the canonical reversibility map $\rho$ on $T^{*}Q$ if and only if the Hamiltonian is reversible, i.e., $H \circ \rho = H$, and in this case the corresponding dynamical system on $\mathcal{M}$ is reversible, i.e., $\rho\circ\varphi^{t} = \varphi^{-t}\circ\rho$.
We say that the nonholonomic system~\eqref{eq:nonhol_system_general} is integrable reversible if its corresponding system on $\mathcal{M}$ is integrable reversible.

\section{Main results} 
\label{sec:main_results}

We consider a family of nonholonomic systems on the configuration space $Q=\R^{3}$, with a distribution of the form
\begin{equation}\label{eq:codistribution}
	\mathcal{D} = \Big\{\, (q_1,q_2,q_3,\dot q_1,\dot q_2,\dot q_3)\in T\R^{3}; f(q_3)\dot q_1 + \dot q_2 = 0  \,\Big\}
\end{equation}
where $f$ is a smooth function on~$\R$, and a Hamiltonian function on $T^{*}\R^{3}$ of the form
\begin{equation}\label{eq:main_perturbed_Hamiltonian}
	H_{0}(\vect{q},\vect{p}) = \frac{1}{2}\sum_{i=1}^{2}\Big( \frac{p_i^{2}}{m_i} + k_i q_i^2  \Big) + F(q_3,p_3) ,
\end{equation}
where $m_i>0$, $k_i>0$, $\varepsilon\geq 0$ are constants and $F$ is a smooth, reversible and regular Hamiltonian function on $T^{*}\R$.
From \eqref{eq:nonhol_system_general} we get the governing equations as
\begin{equation}\label{eq:main_system}
	\begin{aligned}
		\dot q_1 &= \frac{p_1}{m_1}, & 
		\dot p_1 &= -k_1 q_1  + \lambda f(q_3), \\
		\dot q_2 &= \frac{p_2}{m_2}, & 
		\dot p_2 &= -k_2 q_2 + \lambda , \\
		\dot q_3 &= \frac{\pd F}{\pd p_3}, & 
		\dot p_3 &= -\frac{\pd F}{\pd q_3}, \\
		& & 0 &= f(q_3)\dot q_1 + \dot q_2 . 
	\end{aligned}
\end{equation}
From a mechanical point of view, this system consists of two linear oscillators, coupled through a single nonholonomic constraint which depends on the independent Hamiltonian subsystem described by the Hamiltonian $F(q_3,p_3)$.

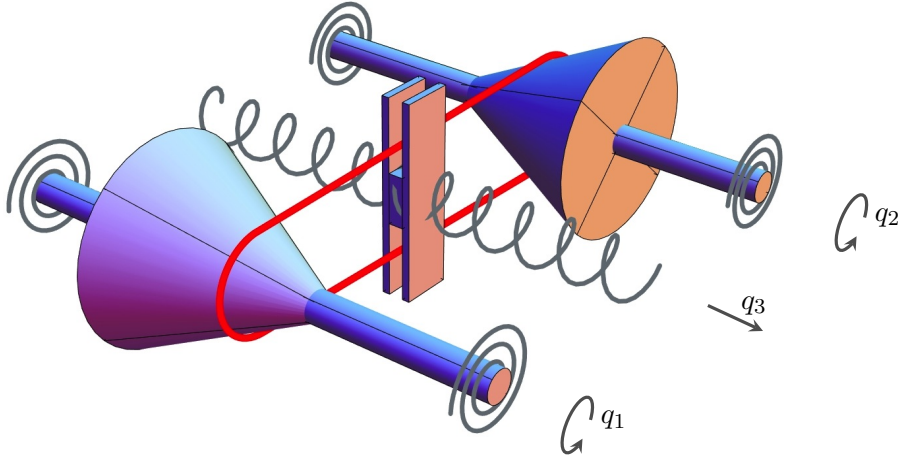
\begin{figure}
	\centering
	\begin{tikzpicture}[
labarrow/.style={very thick,draw,->,color=black!70},
	]
	\node[anchor=south west, inner sep=0] (image) at (0,0) {\includegraphics[width=0.7\textwidth]{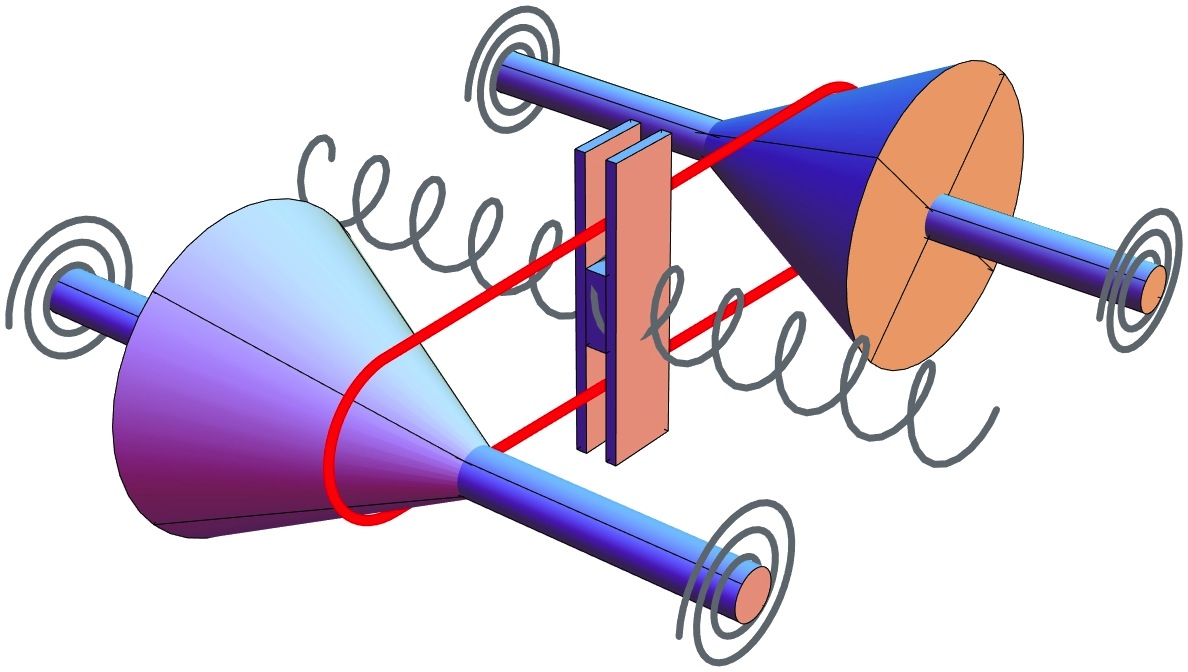}};
	\begin{scope}[x={(image.south east)},y={(image.north west)}]
	\coordinate (q1) at (0.75,0.04) {};
	\coordinate (q2) at (1.1,0.5) {};
	\coordinate (q3) at (0.9,0.3) {};
	\node[right] at (q1) {$q_1$};
	\node[right] at (q2) {$q_2$};
	\path[labarrow] (q3) -- ++(-25:8mm);
	\node[right=3mm,] at (q3) {$q_3$};
	\begin{scope}[x={(1cm,1.2cm)},y={(0cm,1.4cm)},]
	\path[labarrow] (q1)  arc[radius=.2,start angle=-30, delta angle=320];
	\path[labarrow] (q2)  arc[radius=.2,start angle=-30, delta angle=320];
	\end{scope}
	\end{scope}
	\end{tikzpicture}
	\caption{Illustration of a continuously variable transmission gearbox for which the dynamics is governed by the nonholonomic system~\protect\eqref{eq:main_system}.
	The belt between the two cones is translated along the shafts in accordance with the $(q_3,p_3)$~subsystem, thus providing a varying transmission ratio.
	The belt is kept in a plane perpendicular to the shafts, so that the belt keeps a constant length.
	The variables $q_1$ and $q_2$ denote the angular deflections of the shafts.
\todo[author=KM,inline]{Upload the video somewhere on internet and link to it here.}
	}
	\label{fig:elbargetni_machine}
\end{figure}

\begin{remark}
	System~\eqref{eq:main_system} is a simple model for a continuously variable transmission (CVT) gearbox, see~\autoref{fig:elbargetni_machine} for an illustration.
	Indeed, the shafts are attached to spiral springs that are fixed to a chassis, and the system governing $q_3$ is an arbitrary Hamiltonian system with Hamiltonian $F$, which explains the form of the Hamiltonian~\eqref{eq:main_perturbed_Hamiltonian}.
	Now the belt imposes a constraint between the rotation velocities $\dot{q_1}$ and $\dot{q_2}$ which, assuming that both cones have the same size, is determined by $q_3$ as $q_3 \dot{q_1} = (1-q_3) \dot{q_2}$.
	If we assume that $q_3<1$ (which correspond to assuming that the gear ratio is finite), then we obtain a distribution of the form~\eqref{eq:codistribution}.
	As a result, the system~\eqref{eq:main_system} describes a CVT gearbox where the ``driver'' is continuously and periodically shifting according to the independent subsystem $(q_3,p_3)$, and $f(q_3)$ describes the transmission ratio.

	The conceptualisation of CVTs is attributed to \href{http://en.wikipedia.org/wiki/Leonardo_da_Vinci}{Leonardo da Vinci} in 1490.
	In~1879, \href{http://en.wikipedia.org/wiki/Milton_Reeves}{Milton Reeves} invented a CVT for saw milling, which he later used also for cars.
	Today, CVT gearboxes are used in a variety of vehicles, particularly small tractors, snowmobiles, and motorscooters.
\end{remark}

\begin{theorem}\label{thm:main}
	If the $(q_3,p_3)$ subsystem described by the Hamiltonian~$F$ is integrable reversible, then the nonholonomic system~\eqref{eq:main_system} is integrable reversible, or, more precisely, there exists action-angle variables $(a,b,c,\theta,\varphi) \in \R^{3}\times (\Torus)^{2}$ for which the underlying ODE of~\eqref{eq:main_system} takes the form
	\begin{equation}\label{eq:action_angle_main_system}
		\begin{aligned}
			\dot a &= 0, & \dot\theta &= \omega(a), \\
			\dot b &= 0, & \dot\varphi &= \xi(a), \\
			\dot c &= 0, & &
		\end{aligned}
	\end{equation}
	and for which the reversibility map is $(a,b,c,\theta,\varphi) \mapsto (a,b,c,-\theta,-\varphi)$.
\end{theorem}

The proof is given in \autoref{sec:proofs} below.

Consider now a perturbed Hamiltonian $H_\varepsilon = H_0 + \varepsilon G$, where $\varepsilon\geq 0$ and $G$ is a reversible smooth function on $T^{*}\R^{3}$.
We need a Lemma on perturbations of reversible nonholonomic systems, which we specialise to our case.

\begin{lemma}
\label{lem:perturbation}
For small enough $\varepsilon$, the underlying ODE of the nonholonomic system defined by $(\R^{3}, \mathcal{D}, H_{\varepsilon})$ is isomorphic to a reversible perturbation of the underlying ODE of $(\R^3, \mathcal{D}, H_0)$. 
\end{lemma}
\begin{proof}
	First, notice that $H_\varepsilon$ is reversible and for small enough~$\varepsilon$ it is regular.
	Let $\mathcal{M}_\varepsilon$ be the subbundle given by~\eqref{eq:cobundle_distribution} with $H = H_\varepsilon$, and let $X_\varepsilon$ be the corresponding reversible vector field on $\mathcal{M}_\varepsilon$ that generates the dynamics.
	For small enough $\varepsilon$ the map $\frac{\pd H_\varepsilon}{\pd \vect{p}}\colon \mathcal{M}_\varepsilon \to \mathcal{D}$ is a reversible bundle automorphism.
	Thus, $X_\varepsilon$ induces a reversible vector field $X_{0,\varepsilon}$ on $\mathcal{M}_0$ by
	\begin{equation}\label{eq:induced_vf}
		X_{0,\varepsilon} = \Big(\frac{\pd H_0}{\pd \vect{p}}^{-1}\circ \frac{\pd H_\varepsilon}{\pd \vect{p}} \Big)_* X_{\varepsilon}.
	\end{equation}
	This vector field is an $\varepsilon$~perturbation of $X_0$, i.e., $X_{0,\varepsilon} = X_0 + \mathcal{O}(\varepsilon)$.
\end{proof}

An integrable system is called \emph{KAM stable} if ``almost every'' invariant torus, in the sense made precise e.g.\ in \cite{Se1995,Se1998}, is stable under either Hamiltonian or reversible perturbations.
By using \autoref{thm:main} and \autoref{lem:perturbation} we can now use the reversible KAM theorem on the system~\eqref{eq:main_system} to show KAM stability under reversible perturbations.

\begin{corollary}\label{cor:KAM_result}
	If
	\begin{enumerate}
		\item the $(q_3,p_3)$ subsystem is integrable reversible,
		\item $f$ and $F$ are real analytic, and
		\item the frequency functions $\omega(a)$ and $\xi(a)$ in~\eqref{eq:action_angle_main_system} are linearly independent, i.e., there is no constant~$c\in\R$ such that $\omega(\cdot) = c\, \xi(\cdot)$,
	\end{enumerate}
	then the foliation into tori of the underlying ODE of~\eqref{eq:main_system} is KAM stable under reversible perturbations.
\end{corollary}

\begin{proof}
	If the $(q_3,p_3)$ subsystem is integrable reversible it follows from \autoref{thm:main} that the underlying ODE of system~\eqref{eq:main_system} is integrable reversible.
	If $f,F$ are real analytic it follows from the proof of \autoref{thm:main} (see \autoref{sec:proofs} below) that $\omega$ and $\xi$ are real analytic.
	From \autoref{lem:perturbation} it follows that a reversible perturbation of~\eqref{eq:main_system} corresponds to a reversible perturbation of the underlying ODE of~\eqref{eq:main_system}.
	The KAM theorem, as given in~\cite[Thm.\!~1]{Se1998}, then states that the system~\eqref{eq:action_angle_main_system} is KAM stable under reversible perturbations if and only if the image of the map $a\mapsto \big(\omega(a),\xi(a)\big)$ does not lie on a straight line passing through origo, which is exactly the last assumption in \autoref{cor:KAM_result}.
\end{proof}

\begin{remark}
	The system~\eqref{eq:main_system} with $\varepsilon=0$, $k_i=m_i=1$, $F(q_3,p_3) = p_3^{2}/2 + q_3^{2}/2$ and the coupling function $f$ is defined as $f(q_3) := q_3$, was used by \citet[\S~\!5.1]{McPe2006} as a nonholonomic test problem (called the \emph{contact oscillator}) for various numerical discretisations.
	The authors showed that this system has quasiperiodic orbits, and they gave numerical results indicating that quasiperiodicity is preserved by some reversible integrators.
	Moreover, the authors study reversible perturbations of the type described in \autoref{lem:perturbation}, and show by numerical experiments that the same reversible methods seem to preserve the invariant tori foliation, when the perturbations are small.
	By backward error analysis, a reversible method can be interpreted (up to exponentially small terms) as the exact solution of a reversible perturbation of the original nonholonomic system.
	Therefore, \autoref{thm:main} and \autoref{cor:KAM_result} provide a rigorous explanation of the numerical observations by~\citet{McPe2006}.
\end{remark}

\todo[author=OV,inline]{Remark on the form of the codistribution. We assume $u_1 \ud q_1 + u_2 \ud q_2$, with $u_1\neq 0$. This restriction seems to be fundamental, and shapes the geometry of $\Man$.}



\section{Proof of \autoref{thm:main}} 
\label{sec:proofs}

The proof involves the following steps:
\begin{enumerate}
	\item Derive the ordinary differential equation constituting the dynamical system on~$\mathcal{M}_0$.
	\item Derive action-angle variables.
	\item Show that the transformation to action-angle variables preserves reversibility.
\end{enumerate}

\subsection{Ordinary differential equation} 
\label{sub:step_1}

Our aim is to reduce the constrained system~\eqref{eq:main_system} into an ordinary differential equation.
From the governing equation~\eqref{eq:main_system} it follows that~$p_2/m_2=-f(q_3) p_1/m_1$, which implies that the energy can be written
\begin{equation}
	H = \frac{p^2}{2} + \frac{k_1q_1^{2} + k_2q_2^{2}}{2} + F(q_3,p_3),
\end{equation}
where $p:=\frac{1}{\alpha_1(q_3)}\, p_1$, and $\alpha_1(q_3)$ is defined as $\alpha_1(q_3) := \Big(\frac{1+\frac{m_2}{m_1}f(q_3)^2}{m_1}\Big)^{-1/2}$.
We also define $\alpha_2(q_3) := -\frac{m_2}{m_1}f(q_3)\alpha_1(q_3)$, so that $p_2 = \alpha_2(q_3) p$.

By differentiating the identity $p = \frac{1}{\alpha_1(q_3)}p_1$ we obtain
\begin{equation}
\ud p = \alpha_1(q_3) \frac{\ud p_1}{m_1} + \alpha_2(q_3) \frac{\ud p_2}{m_2}
,
\end{equation}
so
\begin{equation}
\dot{p} = \alpha_1(q_3) \frac{\dot{\px}}{m_1} + \alpha_2(q_3)\frac{\dot{\py}}{m_2}
.
\end{equation}

By replacing $\dot{\px}$ and $\dot{\py}$ in the equation above, and by substituting for $p$ in the equation for $\dot{q}_1$ and $\dot{q}_2$,  we now obtain the following system.
\begin{equation}\label{eq:contact_system_reduced}
	\begin{split}
		\dot q_1 &= \frac{\alpha_1(q_3)}{m_1} p \\
		\dot q_2 &= \frac{\alpha_2(q_3)}{m_2} p \\
		 \dot p &= -k_1\alpha_1(q_3) q_1 - k_2\alpha_2(q_3) q_2 \\
		\dot q_3 &= \frac{\pd F}{\pd p_3}(q_3,p_3)\\
		 \dot p_3 &= -\frac{\pd F}{\pd q_3}(q_3,p_3)
	\end{split}
\end{equation}
We have thus obtained an ordinary differential equation describing the system~\eqref{eq:main_system} in the variables $(q_1,q_2,q_3,p,p_3) \in  \R^{3}\times \R^{2}$, which parametrize $\Man$ (which we think of as a trivial fibre bundle over $Q=\R^{3}$).
Since the Hamiltonian is reversible it follows that the canonical reversibility map on $T^{*}\R^{3}$ restricts to the map $\rho\colon (q_1,q_2,q_3,p,p_3)\mapsto (q_1,q_2,q_3,-p,-p_3)$, and that system~\eqref{eq:contact_system_reduced} is reversible with respect to this map.


\subsection{Integrability} 
\label{sub:step_2}

First, since we assumed that the Hamiltonian subsystem~$(q_3,p_3)$ is integrable reversible,
let $a \in \R$ denote the action variable and $\theta \in \Torus$ the angle variable, and~$\omega(a)$ the corresponding angular velocity, so that the canonical reversibility map $(q_3,p_3)\mapsto (q_3,-p_3)$ becomes $(a,\theta)\mapsto (a,-\theta)$.
The system~\eqref{eq:contact_system_reduced}, with obvious abuse notation can now be written
\begin{equation}
\label{eq:reduced_small_action_angle}
	\begin{split}
		\dot q_1 &= \frac{\alpha_1(a,\theta)}{m_1} p \\
		\dot q_2 &= \frac{\alpha_2(a,\theta)}{m_2} p \\
		 \dot p &= -k_1\alpha_1(a,\theta) q_1 - k_2\alpha_2(a,\theta) q_2 \\
		\dot{a} &= 0,\\
		 \dot{\theta} &= \omega(a).
	\end{split}
\end{equation}
Let $\so(3)$ denote the Lie algebra of skew symmetric $3\times 3$~matrices.
The structure of this system becomes evident if we introduce the vector variable $\vect{u}=(\sqrt{{k_1}{m_1}}q_1,\sqrt{{k_2}{m_2}}q_2,p)$ and the~$\so(3)$ valued function
\begin{equation}\label{eq:matrix_fun}
	\matB(a,\theta) = 
	\begin{bmatrix}
		0 & 0 & \sqrt{\frac{k_1}{m_1}}\alpha_1(a,\theta) \\
		0 & 0 &  \sqrt{\frac{k_2}{m_2}}\alpha_2(a,\theta)\\
-\sqrt{\frac{k_1}{m_1}}\alpha_1(a,\theta) & -\sqrt{\frac{k_2}{m_2}}\alpha_2(a,\theta) & 0
	\end{bmatrix}.
\end{equation}
Then, in the variables $(\vect{u},a,\theta)\in \R^{3}\times\R\times \R/\Z$, system~\eqref{eq:reduced_small_action_angle} takes the form
\begin{equation}\label{eq:reduced_skew_form}
	\begin{split}
		\dot{\vect u} &= \matB(a,\theta)\vect{u} \\
		\dot a &= 0,\quad \dot\theta = \omega(a).
	\end{split}
\end{equation}
 
Our aim below is to show that equation~\eqref{eq:reduced_skew_form} is integrable for each fixed value of~$a$, which then proves that equation~\eqref{eq:contact_system_reduced} is integrable.
If $\omega(a)=0$, then $\matB(a,\theta)$ is a constant rotation matrix, so the system has invariant tori for these values of~$a$.
If $\omega(a) \neq 0$ then we can rescale the time variable so that the angular velocity of~$\theta$ is~$1$.
The system~\eqref{eq:reduced_skew_form} can then be written
\begin{equation}\label{eq:reduced_skew_scaled}
	\begin{split}
		\dot{\vect u} &= \matA_a(\theta)\vect{u} \\
		\dot\theta &= 1.
	\end{split}	
\end{equation}
where 
\begin{equation}
\label{eq:defmatA}
\matA_a(\theta) := \matB(a,\theta)/\omega(a)
.
\end{equation}
That this system is integrable is a consequence of the following more general result.

\begin{theorem}
\label{th:integrability}
	Consider a differential equation of the form 
	\begin{equation}
	\label{eq:autoper}
	\begin{split}
		\dot{\vect u} &= \matA\big(\theta\big)\vect u\\
		\dot{\theta} &= 1
	\end{split}
	\end{equation}
	where $(\vect u,\theta) \in \R^n \times \Torus$ are the dependent variables and $\matA \colon \Torus \to \so(n)
$ is a smooth mapping.
	There is a smooth change of variables
	\begin{equation}
	\Psi \colon \R^n \times \Torus \ni (\vect u,\theta)\mapsto (\vect v(\vect u,\theta),\theta) \in \R^{n} \times \Torus
	\end{equation}
	such that system \eqref{eq:autoper} expressed in the new variables $(\vect v,\theta)$ takes the form
	\begin{equation}
	\begin{split}
	\dot{\vect v} &= \matAbar \vect v \\
	\dot{\theta} &= 1
	\end{split}
	\end{equation}
	for a constant matrix $\matAbar \in \so(n)$, whose spectrum $\sigma(\matAbar)$ is such that 
	\begin{equation}
	\label{eq:openspectrum}
	\sigma(\matAbar)\subset [-\ii \pi, \ii \pi]
	.
	\end{equation}
\end{theorem}


\begin{proof}

One defines the flow map $\flow(\thera)$ as the solution operator of the differential equation defined for all $\thera \in \R$ by
\begin{equation}\label{eq:nonautonomous}
	\frac{\ud\vect w(\thera)}{\ud\thera} = \matA(\thera)\vect w(\thera)	
\end{equation}
with initial condition at $\thera = 0$ --- the initial time matters because the differential equation is not autonomous.
This means that if $\vect w$ is a solution of \eqref{eq:nonautonomous}, then
\begin{equation}
\vect w(\thera) = \flow(\thera) \vect w(0) \qquad \forall \tau \in \R
,
\end{equation}
and vice versa.

Since $\matA(\thera)\in\so(n)$ for all $\thera$, the flow map after one period, i.e., $\perPhi$, is an element of $\SO(n)$.
As a result, there exists a matrix $\matAbar \in \so(n)$, with the restriction \eqref{eq:openspectrum} on the spectrum, such that
\begin{equation}
\label{eq:defmatAbar}
\perPhi = \exp(\matAbar)
.
\end{equation}
\newcommand*\bPsi{\overline{\Psi}}
We define the mapping $\bPsi \colon \R^n \times \R \to \R^n \times \R$ by $\bPsi(\vect u,\thera) = \big(\vect v(\vect u,\thera),\thera\big)$ and
\begin{equation}
\vect v(\vect u,\thera) := \exp(\matAbar \thera) \flow(\thera)\inv \vect u
.
\end{equation}

Consider a solution $ \big(\vect u(t),\thera(t)\big)$ of the differential equation 
\begin{equation}
	\begin{split}
		\dot{\vect u}(t) &= \matA\big(\tau(t)\big)\vect u(t) \\
		\dot{\tau}(t) &= 1
	\end{split}
\end{equation}
Define $t_0 = -\thera(0)$.
Clearly we then have $\thera(t) = t - t_0$.
By defining $\vect w(\tau) = \vect u(\tau + t_0)$ we obtain
\(
\vect w'(\tau) =  \dot{\vect u}(\thera + t_0) =  \matA\big(\theta( \thera + t_0)\big) =  \matA(\thera)
\),
so $\vect w$ is a solution of \eqref{eq:nonautonomous}.
As a result, $\vect w(\tau) = \flow(\tau)\vect w(0)$, which, using $\vect u(t) = \vect w\big(\thera(t)\big)$, gives 
\(
\vect u(t) = \flow({t-t_0}) \vect u(t_0) = \flow\big(\thera(t)\big) \vect u(t_0)
\).
We thus obtain that along a solution $\big(\vect u(t),\thera(t)\big)$,
\begin{equation}
\label{eq:phiinvuz}
\flow\big(\thera(t)\big)\inv \vect u(t) = \vect u(t_0)
.
\end{equation}
As a result, we obtain
\begin{equation}
\dot{\vect v}(t) = \matAbar \exp(\matAbar\thera) \vect u(t_0) = \matAbar \vect v(t)
,
\end{equation}
which proves the result for the mapping $\bPsi$.

Recall that, one of the pillars of Floquet's theory is that for any integer $n\in\Z$ and any $\tau\in\R$ we have (\cite[\S\!~28]{Ar2006})
\begin{equation}
\label{eq:Floquet}
\flow(n+\tau) = \flow(\tau)\flow(n)
.
\end{equation}
Now, $\vect v(\vect u,\thera)$ is periodic in $\thera$, of period one, because, due to the Floquet property \eqref{eq:Floquet}, and the definition \eqref{eq:defmatAbar} of $\matAbar$,
\begin{equation}
\exp\big(\matAbar(\thera+1)\big)\flow(\thera+1)\inv = \exp(\matAbar\thera)\exp(\matAbar)\perPhi\inv\flow(\thera)\inv =  \exp(\matAbar\thera)\flow(\thera)\inv
.
\end{equation}
As a result, the mapping $\bPsi$ induces a mapping $\Psi \colon \R^n \times \Torus \to \R^n \times \Torus$ which has the desired properties.
\end{proof}


\subsection{Reversibility} 
\label{sub:step_3}

We equip the space $\R^n\times\Torus$ with a linear involution $R$, defined as
\begin{equation}
\label{eq:Rdef}
R(\vect u,\theta) = (\rho \vect u, -\theta)
,
\end{equation}
for a given linear orthogonal involution $\rho$.
One verifies by a straightforward calculation that the differential equation \eqref{eq:autoper} is reversible with respect to~$R$ if and only if
\begin{equation}
\label{eq:Areversible}
\rho \matA(\theta) \rho = -\matA(-\theta)
.
\end{equation}

\begin{theorem}
Assume that \eqref{eq:Areversible} holds, and that $\theta\mapsto\matA(\theta)$ is periodic of period one.	
Assume further that the spectrum $\sigma(\matAbar)$ of the matrix $\matAbar$ defined in \autoref{th:integrability} fulfils
\begin{equation}
\label{eq:nohalfrotation}
\sigma(\matAbar) \subset (-\ii \pi, \ii \pi)
.
\end{equation}
Then the mapping $\Psi\colon \R^{n}\times\Torus\to\R^{n}\times\Torus$, defined in \autoref{th:integrability}, preserves reversibility, i.e., $R\circ\Psi = \Psi\circ R$.
\end{theorem}


\begin{proof}

Using \eqref{eq:Areversible}, one shows that
\begin{equation}
\rho\flow(-\tau) =  \flow(\tau) \rho
.
\end{equation}
This implies in particular that
\(
\rho \flow(-1) = \perPhi \rho
\),
so using the Floquet property \eqref{eq:Floquet}, we obtain
\begin{equation}
\label{eq:mirrorrot}
\rho \perPhi\inv = \perPhi \rho
.
\end{equation}
Now, recalling the definition of $\matAbar$ in \eqref{eq:defmatAbar}, we notice that \eqref{eq:mirrorrot} implies that
\begin{equation}
\exp(-\matAbar) = \rho\exp(\matAbar)\rho = \exp(\rho\matAbar\rho)
\end{equation}
and, using the assumption on the spectrum of $\matAbar$, and the fact that the eigenvalues of $\rho\matAbar\rho$ are the same as those of $\matAbar$, we deduce that $-\matAbar = \rho \matAbar \rho$.
We therefore obtain
\begin{equation}
\label{eq:commutexprho}
\exp(-\matAbar \thera) \rho = \rho \exp(\matAbar\thera)
.
\end{equation}
By combining \eqref{eq:mirrorrot} and \eqref{eq:commutexprho} we get
\begin{equation}
\begin{split}
\vect v(\rho \vect u, -\thera) &= \exp(-\matAbar\thera)\flow(-\thera)\inv \rho \vect u  \\
&= \exp(-\matAbar\thera) \rho \flow(\thera)\inv \vect u \\
&= \rho \exp(\matAbar \thera) \flow(\thera)\inv \vect u
\end{split}
\end{equation}
so we finally obtain
\begin{equation}
\vect v(\rho \vect u, -\thera) = \rho \vect v
.
\end{equation}
This finishes the proof.
\end{proof}

\begin{remark}
	The condition \eqref{eq:nohalfrotation} is a non-resonance assumption.
	We already know that the matrix $\matAbar$ is chosen to satisfy \eqref{eq:openspectrum}, but we cannot exclude the extreme case of one or more half rotations.
\end{remark}


\subsection{Wrapping up} 
\label{sec:wrapup}

We finish the proof by putting all the pieces together.
Notice that with the canonical mapping $\rho \colon \R^3 \to \R^3$ defined as $(q_1,q_2,p) \mapsto (q_1,q_2,-p)$, the property \eqref{eq:Areversible} is fulfilled for the matrix $\matA_a$ defined in \eqref{eq:defmatA}.
We thus obtain that, in the special case of $\R^3$, the map $\Psi$ defined in \autoref{th:integrability} transforms the original system to an integrable one, and the map $\Psi$ is reversible.
In $\R^3$, all rotations have one axis of rotation, so for all the values of the invariant $a$, we obtain two more invariants, and one angle variable.
In the end, we thus obtain three invariants and two angle variables.


\section*{Acknowledgements} 
\label{sec:acknowledgements}

The research was supported by a Marie Curie International Research Staff Exchange Scheme Fellowship within the 7th European Community Framework Programme and by the Swedish Research Council contract VR-2012-335.



\bibliographystyle{amsplainnat} 
\bibliography{references}


\end{document}